\documentclass[letterpaper, 10 pt, conference]{ieeeconf}

\IEEEoverridecommandlockouts                  
\usepackage{cite}
\usepackage{amsmath,amssymb,amsfonts,bm}
\usepackage{enumerate}
\usepackage{graphicx}
\usepackage{textcomp}
\usepackage{xcolor}
\usepackage{color}
\usepackage{amsthm}
\usepackage{amsmath}
\usepackage{algorithm}
\usepackage{dsfont}
\newcommand\mydots{\hbox to 0.7em{.\hss.\hss.}}
\usepackage{algpseudocode}

\usepackage[labelsep=period]{caption}

\usepackage{tikz}
\usepackage{dsfont}
\newtheorem{thm}{Theorem}

\newtheorem{lem}[thm]{Lemma}

\theoremstyle{definition}
\newtheorem{ass}{Assumption}

\newtheorem{defn}{Definition}
\newtheorem{rem}{Remark}
\newtheorem{exmp}{Example}
\newtheorem{alg}{Algorithm}

\usepackage[utf8]{inputenc}
\usepackage{color}

\usepackage[titles]{tocloft}

\usepackage{hyperref}

\begin{document}

\title{General-Sum Linear Regulator Games for Positive Systems}

\author{Alba Gurpegui$^{*,1}$ , Monika Tomar$^{*,2}$ and Takashi Tanaka$^2$ 
\thanks{This work is partially funded by the Wallenberg AI, Autonomous Systems and Software Program (WASP), the European Research Council (ERC) under the European Union's Horizon 2020 research and innovation programme under grant agreement No 101199738, DARPA COMPASS program under grant agreement HR0011-25-3-0210 and AFOSR DSCT program under grant agreement FA9550-25-1-0347.}
\thanks{$^*$These authors contributed equally. $^1$ A.Gurpegui is with the Department of Automatic Control and the ELLIIT Strategic Research Area at Lund University, Lund, Sweden. $^2$ M. Tomar and T. Tanaka are with the School of Industrial Engineering, School of Aeronautics and Astronautics, Elmore Family School of Electrical and Computer Engineering, Purdue University, West Lafayette, IN, USA. Email: \href{mailto:alba.gurpegui@control.lth.se}{alba.gurpegui\_ramon@control.lth.se}, 
\href{mailto:tomarm@purdue.edu}{tomarm@purdue.edu}, 
\href{mailto:tanaka16@purdue.edu}{tanaka16@purdue.edu}, }
}
\maketitle
\begin{abstract}
This paper studies a continuous-time general-sum non-cooperative game with linear costs, positive linear system dynamics, and elementwise linear input constraints. In the finite-horizon case, we present 
a verification theorem characterizing feedback Nash equilibria, in terms of absolutely continuous solutions of a coupled system of vector-valued ordinary differential equations, realized by time-varying feedback laws. 
Unlike linear-quadratic differential games, whose Riccati-based equilibria scale quadratically with the state dimension, this formulation scales linearly. However, the resulting piecewise-constant feedback saturates between its constraint bounds rather than varying smoothly, and additional mathematical challenges arise when characterizing the solutions of the differential equations, which are generally discontinuous due to the switching nature of the feedback gains. In this work, we study the case where switching occurs only at isolated time instants. In the infinite-horizon case, under stabilizability assumptions, the equilibrium is characterized by coupled vector-valued algebraic equations. For this game, we propose iterative methods to compute both finite and infinite-horizon equilibria. The approach is illustrated through a large-scale pollution game.
\end{abstract}
\vspace{-3mm}
\section{Introduction}
\vspace{-2mm}
Continuous-time dynamic game theory provides a natural framework for modeling interactions among strategic agents whose decisions evolve over time. This framework finds applications across a wide range of fields, including economics~\cite{app_econom}, ecology~\cite{app_eeco} and epidemiology~\cite{app_epi}. 
In non-cooperative settings, different solution concepts are available, but the Nash equilibrium solution remains the most prominent. This solution is defined by the property that no player can reduce its cost through unilateral deviation. Nash equilibria can be defined in different ways: in open-loop form, where control inputs depend only on time and initial conditions, or in feedback form, where they depend on the current state. In general, these notions are not equivalent, although they coincide for certain classes of games~\cite{Fershtman, ref_open_feedback}. Equilibria in non-cooperative games can be classified according to their robustness and time-consistency. Feedback Nash equilibria are generally considered strongly time consistent, as they allow players to react to deviations in the state trajectory regardless of past policies~\cite{Basar_eq}. These equilibria play a central role in the differential games literature~\cite{basar}, particularly in linear-quadratic LQ games~\cite{basar2, ref9, ref11_engwerda, takashi_ref}, including the Riccati-based formulations for positive linear systems~\cite{azebedo-jank} and associated iterative solvers~\cite{Ivanov}. 
Positivity of the system dynamics arises naturally in game-theoretic settings such as evolutionary and population dynamics~\cite{evol_games} and pollution games~\cite{yu2024research}, where states (populations, pollutant stocks) are intrinsically nonnegative and costs are linear, as the profits from emitting and the costs of pollution control grow in proportion to the amount involved. In this paper, we exploit this positivity property and build on prior work in zero-sum dynamic games, formulated as minimax optimal control problems with linear costs, linear positive dynamics, and elementwise linear input constraints, which has been studied both in discrete~\cite{AlbaEmmaAnders} and continuous time~\cite{LR_paperII}. This class, referred to as the Minimax Linear Regulator (LR) problem, admits explicit solutions for all nonnegative initial conditions under appropriate assumptions. 
Inspired by these results, we extend the minimax LR framework to a general-sum non-cooperative dynamic game setting and study the corresponding feedback Nash equilibria. 
The contributions of this work are summarized below.
\begin{enumerate}[C1]
    \item We propose a general-sum non-cooperative dynamic game with linear costs, positive linear dynamics and elementwise linear constraints on the inputs.
    \item We present a verification theorem characterizing 
    feedback Nash equilibria given by state-feedback policies with piecewise-constant gains, in terms of absolutely continuous solutions of coupled vector-valued ordinary differential equations (ODEs).
    \item Under appropriate stabilizability assumptions, we present the stationary feedback Nash equilibria, which are given by static state-feedback laws and characterized by coupled vector-valued algebraic equations.
    \item We develop iterative algorithms for computing the equilibria of the considered finite and infinite-horizon settings and illustrate our results on a large-scale pollution game.
\end{enumerate}
\vspace{-3mm}
\subsection{Notation}Let $\mathbb{R}_{+}$ denote the set of nonnegative real numbers, $\mathbb{R}^{n}$ the $n$-dimensional Euclidean space and $\mathbb{R}^{n}_{+}$ the positive orthant. $\mathbb{R}^{m \times n}$ denotes the set of $m \times n$ real matrices. Any vector is, by default, a column vector. $\mathds{1}$ denotes the vector of ones of appropriate dimension, and $I_n$ denotes the $n \times n$ identity matrix. For a vector $x$, $\mathrm{diag}(x)$ denotes the diagonal matrix with entries of $x$ on the diagonal. $\left | X \right |$ denotes the matrix obtained by replacing the elements of a real matrix $X$ with their absolute values. Inequalities between matrices or vectors are understood elementwise; in particular, $X \geq 0$ ($X\leq0)$ means that all entries of $X$ are nonnegative (nonpositive). The signum of a scalar is defined as the set-valued map 
\vspace{-5pt}
$$\scalebox{.8}{$\operatorname{sign}(x) =$}
    \scalebox{.8}{$\ensuremath{
    \begin{cases}
        \{-1\} & \text{if } x < 0\\
        [-1,+1] & \text{if } x = 0\\
        \{+1\} & \text{if } x > 0
    \end{cases}}$}.$$
\section{Preliminaries}
Dynamical systems that leave the positive orthant of the state space invariant are called positive systems (related concepts include compartmental systems and monotone systems \cite{Angeli}), and have attracted special attention in the control theory literature. 
A linear dynamical system $\dot x(t)=Ax(t)$ is called \emph{positive} if entry-wise nonnegativity of $x(0)$ implies entry-wise nonnegativity of $x(t)$ for all $t\geq 0$. It is well known that this holds if and only if $A$ is a Metzler matrix (i.e., a square matrix with non-negative off-diagonal entries) ~\cite[Thm. 2.5]{Kaczorek}.
Motivated by this property, we consider the Linear Regulator (LR) problem, an optimal control problem with linear nonnegative cost, positive linear dynamics and elementwise linear input constraints
\begin{align}\label{LR_optprob}
    &\scalebox{.95}{$\underset{\mu}{\min} \hspace{1mm} \int_{0}^{T} \left [ s^{\top}x(\tau) +r^{\top}u(\tau)\right ]d\tau$} \notag \\
            &~\scalebox{.95}{$\mathrm{s.t} \hspace{3mm} \dot x(t)=Ax(t)+Bu(t), \hspace{1mm} x(0)=x_{0}$} \\
        &~~~~~~\hspace{0.7mm}\scalebox{.9}{$u(t)=\mu(x(t)), \hspace{1mm} \left | u \right | \leq E x$}, \notag
\end{align}
where the vector $x(t)$ denotes the time-varying state, $ u(t)$ the control input, $A \in \mathbb R^{n\times n}$, $B \in \mathbb R^{n\times m}$, $E \in \mathbb R^{m\times n}_+$, $s\in \mathbb R^n$, and $r \in \mathbb R^m$. 

Remarkably, explicit solutions for this problem are derived in~\cite{LR_paperII} under the assumption $A-|B|E$ is Metzler and $s-E^{\top}|r|>0$, motivated in Remark~\ref{rem_ass}. Under these conditions, the explicit solution of the problem~\eqref{LR_optprob} is characterized by a vector $p(t)\in \mathbb R^n_+$ satisfying 
\begin{align}\label{LR_ODE}
        \scalebox{.95}{$-\dot p(t)= s+ A^{\top}p(t)- E^{\top} |r+B^{\top}p(t)|, \hspace{2mm} p(T)=0$}.
\end{align}
In fact, if the LR problem has a solution, by Picard-Lindelöf~\cite{kelley2010theory} the solution is unique, the optimal cost is given by  $p(0)^{\top}x_0$ and the optimal control policy is, among all potentially nonlinear policies, a time-varying policy $u^*(t)=-K(t)x(t)$ with $K(t)\in \hspace{0.5mm}\mathrm{diag}(\mathrm{sign}(r+B^{\top}p(t)))E.$ We refer to the entries ($r+B^{\top}p(t)$) as the \textit{input gradients}. The optimal policy defined above is not necessarily unique. Specifically, when an input gradient vanishes at some index $i$, all feedback gain matrices in the set 
\begin{align*}
    \scalebox{.91}{$\mathcal{K}=\left\{DE \hspace{1mm} | \hspace{1mm} D_{ii}\in \left[ -1, 1\right], D_{jj} \in \operatorname{sign}([r+p(t)^{\top}B]_j) \hspace{1mm}\mathrm{ for } \hspace{1mm} j\neq i \right\}$}
\end{align*}    
    lead to the same and unique solution $p(t)$ of the ODE~\eqref{LR_ODE}. 
Importantly,  the optimal policy inherits the sparsity structure of the $E$ matrix, which is determined by the problem designer and may capture limitations in actuation
and sensing. 

In the infinite-horizon case, we additionally require the optimal solutions to ensure closed-loop stability. A notion of stabilizability within the constrained input class $|u|\leq Ex$, is defined in~\cite{LR_paperII} as $E$-stabilizability.
\begin{defn}[$E$-stabilizability]\label{Estab}
    Let $A \in \mathbb{R}^{n\times n}$, $B \in \mathbb{R}^{n\times m }$ and $E  \in \mathbb{R}_{+}^{m \times n}$. The pair $(A,B)$ is $E$-stabilizable if there exists a feedback law $u=-Kx$ with $\left|u \right|\leq Ex$ such that $A-BK$ is Hurwitz. 
\end{defn}
The following Lemma characterizes the infinite-horizon optimal value of the LR problem~\eqref{LR_optprob} as $T\rightarrow \infty$.
\begin{lem}\label{LR_inf_thm}
    Let  $A\in \mathbb{R}^{n\times n}$, $B\in \mathbb{R}^{n\times m }$, $E  \in \mathbb{R}_{+}^{m \times n}$, $s \in \mathbb{R}^{n} $, $r\in \mathbb R^m$. Suppose that the pair $(A,B)$ is $E$-stabilizable and $A-|B|E$ is Metzler, $s-E^{\top}|r| >0$ holds. 
    Then, as $T\rightarrow \infty$, problem~\eqref{LR_optprob} has a finite minimum for every $x_0$ if and only if
\vspace{-7pt}
\begin{align}\label{ARE_LR}
    0&= s+A^{\top}p-K^{\top} (r+B^{\top}p)  
\end{align}
holds with $K \in \hspace{0.5mm}\mathrm{diag}\left(\mathrm{sign}( r+B^{\top} p)\right)E$.  
Moreover, the feedback law $u^*(t) := -Kx$ is optimal and $E$-stabilizing and the optimal cost is given by $J(x_0, u^*)= p^{\top}x_0$.
\end{lem}
\begin{proof}
This theorem follows from~\cite[Cor. 7]{LR_paperII}.
\end{proof}

It is worth highlighting that the algebraic equation~\eqref{ARE_LR} can be reformulated equivalently as a linear program~\cite[Thm. 12]{LR_paperII}, leveraging the scalability potential of this optimal control problem class.
\section{Problem Formulation}
Consider a continuous-time system with $N>1$ players described by the differential equation
\begin{align}\label{dynamics}
    \dot x(t)= Ax(t)+ \textstyle\sum\nolimits_{i=1}^N B_i u_i (t), \hspace{1mm} x(0)=x_0
\end{align}
for $t\in [0, T]$, where $x_0 \in \mathbb R^n_+$, $x(t)\in \mathbb R^n$ is the state variable of the system and $u_i(t)\in \mathbb R^{m_i}$ is the control input of the $i$-th player for $i=1,\mydots, N$ satisfying $|u_i|\leq E_ix$, $E_i\in \mathbb R^{{m_i} \times n}_+$. Let $A\in \mathbb R^{n\times n} $ be a Metzler matrix, and $B_i\in \mathbb R^{n \times m_i}$ for all $i$. Each player $i=1,\mydots, N$ has an associated cost functional
\begin{align}\label{cost}
    \scalebox{.94}{$C_i(t, x, u_1,\mydots,u_N)= \int_t^T \left( s_i^{\top}x(\tau)+ r_i^{\top} u_i(\tau) \right) d\tau$}
\end{align}
where $x(s), \hspace{1mm} t\leq s \leq T$ on the right-hand side is defined by \eqref{dynamics} with the initial condition $x(t)=x\in \mathbb R^n_+$, the input $u_i(t)\in \mathbb R^{m_i}$, $i=1,\mydots,N$ and $s_i \in \mathbb R^n_+$, $r_i\in \mathbb R^{m_i}$. We formulate this setting as a non-cooperative differential game, where each player seeks to minimize its individual cost functional, subject to the following assumptions.
\begin{ass}\label{ass_pos_dyn}
    $A-\sum_{i=1}^N |B_i|E_i$ is Metzler.
\end{ass}
\begin{ass}\label{ass_nonneg_cost}
    $s_i-E^{\top}_i|r_i|>0$ for all $i$.
\end{ass}
\begin{rem}\label{rem_ass} Assumption~\ref{ass_pos_dyn} ensures the invariance of the positive orthant under the system dynamics~\eqref{dynamics} and Assumption~\ref{ass_nonneg_cost} guarantees that each cost functional~\eqref{cost} is nonnegative and bounded from below.
\end{rem}
We refer to the resulting non-cooperative differential game as the Linear Regulator (LR) general-sum game. 
In this article, we focus on feedback Nash-equilibrium solutions, as recalled in the following definition~\cite{basar}.
\begin{defn}[Feedback Nash equilibrium] Consider the dynamic game defined by~\eqref{dynamics} and~\eqref{cost}. The set of feedback control laws $(\mu_{1}^*, \mydots, \mu_{N}^*)$ is a feedback Nash equilibrium if 
\begin{align*}
    &\scalebox{.95}{$C_i(t, x, \mu_{1}^*,\mydots, \mu_{i}^*, \mydots, \mu_{N}^*)\leq C_i(t, x, \mu_{1}^*,\mydots, \mu_{i}, \mydots, \mu_{N}^*)$}
\end{align*}
for all pairs \scalebox{.95}{$(t,x)\in [0,T]\times \mathbb R^n_+$}, where \scalebox{.95}{$C_i$} is the cost functional~\eqref{cost}.
\end{defn}
\section{Main Results} 
\subsection{Finite Horizon}
In this section, we study the non-cooperative differential game over a fixed horizon \scalebox{.95}{$[0,T]$}, in which each player chooses an admissible state-feedback law subject to elementwise linear input constraints. 
The corresponding feedback Nash equilibria are characterized by a coupled system of vector-valued ordinary differential equations.

We recall the notion of absolutely continuous (AC) function before stating the main result of this section.
\begin{defn}
    The function \scalebox{.95}{$\gamma:[a,b]\rightarrow\mathbb R$} is absolutely continuous if, for all \scalebox{.95}{$\epsilon>0$}, there exists \scalebox{.95}{$\delta>0$} such that for each finite collection \scalebox{.95}{$\big\{(a_1, b_1),\mydots, (a_n, b_n)\big\}$} of disjoint open intervals contained in \scalebox{.95}{$[a,b]$} with \scalebox{.95}{$\sum_{i=1}^n(b_i-a_i)<\delta$}, it follows that \scalebox{.95}{$\sum_{i=1}^n|\gamma(b_i)-\gamma(a_i)|<\epsilon$}.  
\end{defn}
The following theorem characterizes the feedback Nash equilibrium in terms of absolutely continuous solutions of a coupled ODE system, under the assumption that the entries of the input gradients 
\scalebox{.95}{$r_i+B_i^{\top}p_i$}, \scalebox{.95}{$i=1,\mydots, N$} are nonzero almost everywhere. 
\begin{thm}\label{finite-hor}
    Suppose that Assumption~\ref{ass_pos_dyn} and Assumption~\ref{ass_nonneg_cost} hold.  Suppose also that there exist absolutely continuous functions $p_i:[0,T]\rightarrow \mathbb R^n, \hspace{1mm} i=1,\mydots, N$ satisfying
    \begin{align}\label{ODEs}
    &\scalebox{.94}{$-\dot p_i(t) =s_i+ A^{\top}p_i(t)-K_i(t)^{\top}(r_i+ B_i^{\top}p_i(t))$}\\
    &~~~~~~~~~~~~~~~~~~~~~~~~~~~~~~~~~~~\scalebox{.94}{$-\textstyle\sum\nolimits_{j \neq i} 
    K_j(t)^{\top}B_j^{\top}p_i(t)$} \notag
\end{align} 
almost everywhere, with $p_i(T)=0$, $K_i(t)=\mathrm{diag}(\mathrm{sign}(r_i+B^{\top}_i p_i(t)))E_i$ and the entries of the input gradients of each player $i=1,\mydots, N$ satisfy $r_i+B_i^{\top}p_i(t) \neq 0$ for almost every $t\in [0,T]$. Then the feedback laws   \begin{align}\label{Nash_fin}
        \scalebox{.95}{$u_{i}^*(t)= -K_i(t)x(t)$}
    \end{align}
define a feedback Nash equilibrium for the dynamic game \eqref{dynamics} and \eqref{cost}. Moreover, the cost incurred by player $i$ is given by $p_i(0)^{\top}x_0$, $i=1,\mydots, N$.
\end{thm}
\begin{proof}
We prove the result for the two-player case for simplicity, the general case follows analogously. Assume that the coupled ODE system~\eqref{ODEs} admits a pair of AC solutions $(p_1(t), p_2(t))$, $t\in [0,T]$, satisfying~\eqref{ODEs} almost everywhere and such that every entry of \scalebox{.92}{$r_i+B_i^{\top}p_i(t) \neq 0$} for almost every \scalebox{.92}{$t\in [0,T]$}. We verify that the feedback laws defined in~\eqref{Nash_fin} constitute a feedback Nash equilibrium. Fix any pair \scalebox{.92}{$(t,x)\in [0,T]\times \mathbb R^n$} and the second player's equilibrium strategy \scalebox{.92}{$u_2^*=-K_2(t)x$}. The optimization problem of Player 1 has a running cost \scalebox{.92}{$g_1(x, u_1)=s_1^{\top}x+r_1^{\top}u_1$}, dynamics \scalebox{.92}{$f_1(x, u_1)=(A-B_2K_2) x+B_1u_1$} and an admissible set \scalebox{.92}{$U_1(x)=\big\{u_1 \hspace{0.1mm}: \hspace{0.1mm}|u_1|\leq E_1 x \big\}$}.  Define the candidate value function \scalebox{.92}{$V_1(\tau,x)=p_1(\tau)^{\top}x$}. Since $p_1$ is AC by assumption, the right hand side of $f_1$ is measurable in $\tau$ and continuous in $x$, and for every compact set \scalebox{.95}{$D \subset [0,T] \times \mathbb R^n\times \mathbb R^m$}  with \scalebox{.92}{$|x|\leq \alpha \mathds 1$}, for some \scalebox{.9}{$\alpha>0$}, \scalebox{.92}{$|f_1|\leq \kappa \hspace{0.5mm} \alpha \mathds 1$} for all $t$ and \scalebox{.92}{$|f_1(x,u_1)-f_1(y, u_1)|\leq \kappa |x-y| $} with \scalebox{.92}{$\kappa=(\big\|A\big\|+\big\|B_2\big\| \big\|E_2\big\|+\big\|B_1\big\|\big\| E_1 \big\|)$}. By the Carathéodory existence and uniqueness theorems~\cite[Thm. 5.1, Thm. 5.3]{caratheodory} for every admissible input, the system \scalebox{.92}{$f_1(x,u_1)$} admits a unique Carathéodory solution $x$ which is AC, hence \scalebox{.92}{$V_1(\tau,x)$} is AC. 
By the Fundamental Theorem of Calculus for AC functions~\cite[Thm. 7.18]{Rudin} 
\begin{align*}
    \scalebox{.9}{$p_1(T)^{\top}x(T)-p_1(t)^{\top}x(t)=\int_t^T (\dot p_1(\tau)^{\top}x(\tau)+p_1(\tau)^{\top}\dot x(\tau))d \tau$}.
\end{align*}
Since $p_1(T)=0$, substituting the dynamics $\dot x=f_1(x, u_1)$ and by definition of $\dot p_1$ in~\eqref{ODEs} gives
\begin{align*}
    &\scalebox{.87}{$-p_1(t)^{\top}x(t)=\int_{t}^T\big[-s_1^{\top}x(\tau)+(r_1+B_1^{\top}p_1(\tau))^{\top}K_1x(\tau)$}\\\
    &~~~~~~~~~~~~~~~~~~~~~~~~~~~~~~~~~~~~~~~~~~~~~~\scalebox{.87}{$+p_1(\tau)^{\top}B_1u_1(\tau)\big]d\tau$.}
\end{align*}
Adding and subtracting $r_1^{\top}u_1$ inside the integral yields
\begin{align*}
    &\scalebox{.87}{$p_1(t)^{\top}x(t)=\int_t^T\big[(s_1^{\top}x(\tau)+r_1^{\top}u_1(\tau))-r_1^{\top}u_1(\tau)$}\\
    &~~~~~~~~~~~~~~\scalebox{.87}{$-p_1(\tau)^{\top}B_1u_1(\tau)-(r_1+B_1^{\top}p_1(\tau))^{\top}K_1(\tau)x(\tau) \big]d \tau$.}
\end{align*}
Equivalently,
\begin{align}\label{C_aux}
    &\scalebox{.87}{$C_1(t, x, u_1, u_2^*)=p_1^{\top}(t)x(t)$}\\
    &~~~~~~~~~~~~~~~~~~\scalebox{.87}{$+\int_t^T [(r_1+B_1^{\top}p_1(\tau))^{\top}(u_1(\tau)+K_1(\tau)x(\tau))]d\tau.$}\notag
\end{align}
Under the admissible set \scalebox{.92}{$-E_1x(t)\leq u_1(t)\leq E_1x(t)$}, \scalebox{.92}{$u_1(t)$} seeks to minimize \scalebox{.92}{$C_1$}. By Assumption~\ref{ass_pos_dyn} the closed-loop trajectory satisfies $x(t)\geq 0$. Since \scalebox{.92}{$r_1+B_1^{\top}p_1 \neq 0$} for almost every \scalebox{.92}{$\tau \in [t, T]$}, the input \scalebox{.92}{$u_1^*(t)=-\text{diag}(\text{sign}(r_1+B_1^{\top}p_1))E_1 x(t)$} is the pointwise minimizer of \scalebox{.92}{$(r_1+B_1^{\top}p_1(t))^{\top}u_1$} over the box \scalebox{.92}{$[-E_1x(t), E_1x(t)]$} where \scalebox{.92}{$(r_1+B_1^{\top}p_1(t))^{\top}(u_1(t)+K_1(t)x(t))\geq 0$.} 
Integrating over $[t,T]$ yields \scalebox{.92}{$C_1(t, x, u_1, u_2^*)\geq p_1^{\top}(t)x(t)$}, and equality is achieved when \scalebox{.92}{$u_1=u_1^*$}. The same logic applies to player 2. Hence, by~\cite[Definition 6.2]{basar} \scalebox{.92}{$(u_1^*, u_2^*)$} is a feedback Nash equilibrium.
\end{proof}
\begin{rem}\label{rem:measure}
    Note that the right-hand side of~\eqref{ODEs} is piecewise affine in \scalebox{.92}{$p_i$}, \scalebox{.92}{$i=1,\mydots,N$}. On each subinterval where every entry of the input gradients \scalebox{.92}{$r_i+B^{\top}_ip_i(t)\neq 0$}, the matrices \scalebox{.92}{$K_i$} are constant and the system reduces to an affine ODE. The right-hand side is then measurable in $t$, continuous in $p_i$ and bounded by \scalebox{.92}{$\big\|s_i\big\|+\big\|E_i\big\|\big\|r_i\big\|+(\big\| A\big\|+\textstyle\sum\nolimits_{j=1}^N \big\| B_j \big\| \big\|E_j\big\|)\beta$} on every compact set with \scalebox{.92}{$|p_i|\leq \beta$}, for some \scalebox{.92}{$\beta>0$}, satisfying the Carathéodory conditions, which imply the existence of a solution~\cite[Thm. 5.1]{caratheodory}. Uniqueness in each subinterval holds since the right-hand side is Lipschitz in $p_i$ with constant \scalebox{.92}{$\kappa=(\big\| A\big\|+\textstyle\sum\nolimits_{j=1}^N \big\| B_j\big\| \big\|E_j \big\|)$}. Global uniqueness on $[0,T]$ then follows by induction over the non-switching intervals. 
\end{rem}
The following example illustrates a solution that is not covered by Theorem~\ref{finite-hor}, where \scalebox{.92}{$r_i+B_i^{\top}p_i(t)=0$} on an interval \scalebox{.92}{$I \subseteq [0, T]$} of positive measure.
\begin{figure}[h]
    \centering
    \includegraphics[width=0.37\textwidth]{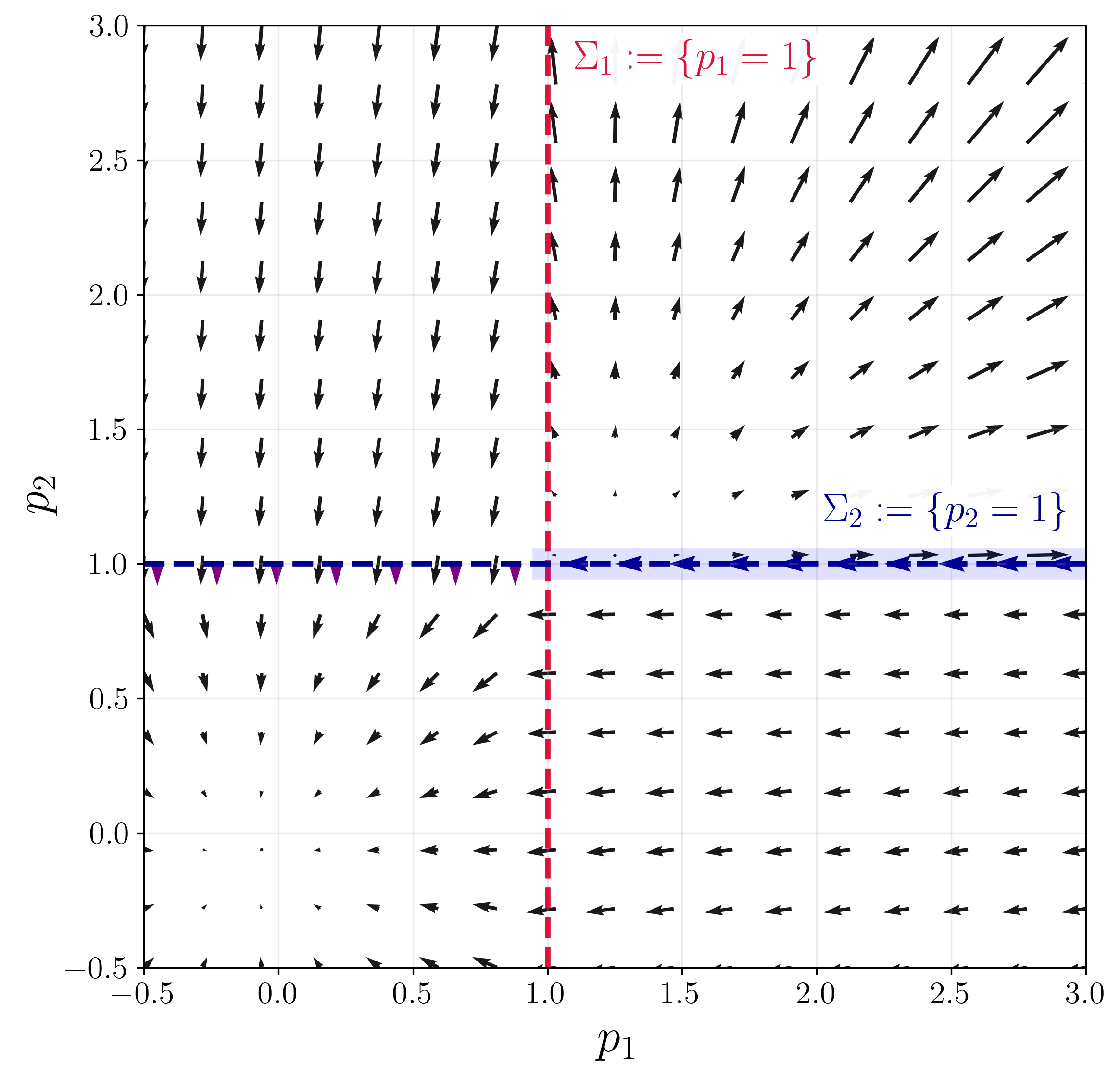}
    \caption{Phase diagram of Example 1 illustrating the switching surfaces $\Sigma_1=\left\{ p_1=1\right\}$ and $\Sigma_2=\left\{ p_2=1\right\}$.}
    \vspace{-20pt}
    \label{fig_exmp_diff_inc}
\end{figure}
\begin{exmp}\label{exmp_filippov}
     Consider \scalebox{.94}{$A=0$}, \scalebox{.94}{$B_1=B_2=1$}, \scalebox{.94}{$E_1=0.5$}, \scalebox{.94}{$E_2=0.4$}, \scalebox{.94}{$r_1=r_2=-1$}, \scalebox{.94}{$s_1=0.6$}, \scalebox{.94}{$s_2=0.5$}, \scalebox{.94}{$T=5$}, \scalebox{.94}{$p_i(T)=0$}. The set of coupled ODEs~\eqref{ODEs} becomes
\begin{align}\label{ODEs_exmp}
    \scalebox{.94}{$\dot {p_1}(t)=-0.6+(p_1(t)-1)K_1(t)+p_1(t)K_2(t)$}\\
    \scalebox{.94}{$\dot{p_2}(t)=-0.5+(p_2(t)-1)K_2(t)+p_2(t)K_1(t)$}
    \label{ODEs_exmp2}
\end{align}
with \scalebox{.94}{$K_1(t)\in 0.5 \hspace{0.5mm}\mathrm{sign}(p_1(t)-1)$}, \scalebox{.94}{$K_2(t)\in 0.4\hspace{0.5mm} \mathrm{sign}(p_2(t)-1).$} Assume \scalebox{.94}{$p_2(t)=1$} and \scalebox{.94}{$p_1(t)\geq 1$} on an interval \scalebox{.94}{$I \subseteq [0, T]$} so that the input gradient is zero on a set of positive measure. 
Recall that this case is not covered by Theorem~\ref{finite-hor}, where it is assumed to be nonzero almost everywhere. From the second equation~\eqref{ODEs_exmp2} \scalebox{.94}{$\dot p_2(t)=0=-0.5+K_1(t)$}, so \scalebox{.94}{$K_1(t)=0.5$}. Then \scalebox{.94}{$\dot p_1=-1.1+(0.5+K_2(t))p_1$}. Since \scalebox{.94}{$K_2(t)\in [-0.4, 0.4]$}, define \scalebox{.94}{$\beta(t):=0.5+K_2(t)\in [0.1, 0.9]$}. Thus, any trajectory satisfying \scalebox{.94}{$p_2=1$} and \scalebox{.94}{$\dot p_1(t)=-1.1+\beta(t)p_1(t)$} with \scalebox{.94}{$\beta(t)\in [0.1, 0.9]$} is admissible. Moreover, because along \scalebox{.94}{$p_2(t)=1$} the value of \scalebox{.94}{$K_2(t)$} is not uniquely defined, the solution fails to be unique and the dynamics are more appropriately described by a differential inclusion~\cite{Cortes, vinter_siam}, which is beyond the scope of this work. Note that on $I$, different choices of \scalebox{.94}{$K_2(t)$} correspond to different trajectories for \scalebox{.94}{$p_1$} through the coupling term \scalebox{.94}{$p_1(t)K_2(t)$}, so non-uniqueness on one player's feedback propagates to the other player's ODE. This behavior is illustrated in Figure~\ref{fig_exmp_diff_inc}, which shows the phase diagram of the system.
\end{exmp}
For approximating the backward ODE system, we next introduce a fully implicit Backward Euler active-set scheme. 
\begin{algorithm}[H]
\caption{Backward Euler with Active-Set Iteration}
\label{alg:BE-AS}
\begin{algorithmic}
\Require  \scalebox{.9}{$T,M,P^T,\ell_{\max}, \textcolor{red}{\epsilon}$ with $\Delta t=T/M$}
\State  \scalebox{.9}{$P^M \gets P^T $, \quad $S_i^M \gets \mathrm{diag}(\mathrm{sign}(r_i+B_i^\top p_i^M))$}

\For{ \scalebox{.9}{$k=M,\mydots,1$}}
\State  \scalebox{.9}{$P^{(0)} \gets P^k$, \quad $S^{(0)} \gets S^k$, \quad $\ell \gets 0$}
\Repeat \hspace{2mm}  \scalebox{.9}{Solve implicit Euler for $P^{(\ell+1)}$}:
\State  \scalebox{.9}{$Q^{(\ell)} \gets I-\Delta t\!\left(A^\top-\sum_{j=1}^{N}E_j^{\top}S_j^{(\ell)}B_j^\top\right)$}
\For{$i=1,\mydots,N$}
\State \scalebox{.9}{$p_i^{(\ell+1)} \gets (Q^{(\ell)})^{-1}\!\left[p_i^k+\Delta t(s_i-E_i^\top S_i^{(\ell)}r_i)\right]$}
\State  \scalebox{.9}{$S_i^{(\ell+1)} \gets \mathrm{diag}(\mathrm{sign}(r_i+B_i^\top p_i^{(\ell+1)}))$}
\EndFor
\State  \scalebox{.9}{$\ell \gets \ell+1$}
\Until{\scalebox{.9}{$S^{(\ell)}\in\{S^{(0)},\ldots,S^{(\ell-1)}\}$} or $\ell=\ell_{\max}$}
\State  \scalebox{.9}{\textbf{if} $S^{(\ell)}\neq S^{(\ell-1)}$ and $\mathrm{Res}(P^{(\ell)}, S^{(\ell)})>\epsilon$, \textbf{return} \texttt{Fail}}
\State \scalebox{0.9}{$P^{k-1} \gets P^{(\ell)}$, \quad $S^{k-1} \gets S^{(\ell)}$}
\EndFor
\State \Return  \scalebox{.9}{$\{P^k,S^k\}_{k=0}^M$}
\end{algorithmic}
\end{algorithm}
\begin{alg}
    As mentioned in Remark \ref{rem:measure} for the Carathéodory notion of solution, we assume that switching of the active set occurs only at isolated times forming a set of measure zero. Let \scalebox{.95}{$S_i(t) = \mathrm{diag}(\mathrm{sign}(r_i+B^{\top}_i p_i(t)))$}. Since the discrete active set \scalebox{.95}{$S_i(t)$} depends on the implicitly unknown state \scalebox{.95}{$p_i(t)$}, we employ an implicit Backward Euler discretization coupled with a fixed-point active-set iteration. At each time step, the algorithm iterates between solving for the $p$-vector and updating the control signs. Repeated sign patterns detect cycling; the iterate is accepted if the discrete residual is below tolerance, otherwise the computation is terminated and may be repeated with a smaller time step.
\end{alg}
\subsection{Infinite Horizon}
In the infinite-horizon setting, we restrict attention to constant linear state-feedback strategies. This is motivated by the requirement of closed-loop stability to guarantee finiteness of the cost, as well as by the fact that the single-player LR problem admits a static optimal solution. 

Denote $u_{K}(t)=-Kx(t)$, $K\in \mathbb R^{m\times n}$.
In this subsection we are concerned with the dynamic game constituted by the dynamics~\eqref{dynamics} and the cost functions~\eqref{cost} as $T\rightarrow \infty$, i.e.
\begin{align}\label{cost_inf}
    \scalebox{.95}{$J_i(x_0, u_1,\mydots, u_N)
= \int_0^{\infty} \left( s_i^{\top}x(t)+ r_i^{\top} u_{K_i}(t) \right) dt$}.
\end{align}
under Assumptions~\ref{ass_pos_dyn} and~\ref{ass_nonneg_cost}. 

In this case, a set of feedback strategies is admissible if it belongs to the subset of state-feedback control laws that render the zero equilibrium of system~\eqref{dynamics} asymptotically stable. In this context, the notion of $E$-stabilizability in Definition~\ref{Estab}, extended to the dynamic game setting, is relevant and necessary for the Nash equilibrium analysis.
\begin{ass}
    We shall limit our set of permitted controls to the constant feedback strategies which are stabilizing. \begin{align}\label{stab_set}
    \scalebox{.94}{$\mathcal K_N:= \Big\{ (K_1,\mydots , K_N): \hspace{0.1mm} A-\textstyle\sum\nolimits_{i=1}^NB_iK_i \hspace{1mm} \text{is} \hspace{1mm}\text{Hurwitz}
    \Big\}$}
\end{align}
\end{ass}
\begin{rem}\label{stblrem}
To verify the \scalebox{.94}{$(E_1,\mydots, E_N)$}-stabilizability of the tuple \scalebox{.94}{$(A,B_1,\mydots, B_N)$} with \scalebox{.94}{$A$} being Metzler, by~\cite[Lem. 4]{LR_paperII} it is necessary and sufficient to verify the feasibility of
\begin{align*}
         \scalebox{.94}{$Ax + \textstyle\sum\nolimits_{i=1}^NB_i u_i \le - \mathds 1, \quad -E_ix \le u_i \le E_ix, \hspace{1mm} \forall i.$}
\end{align*}
\end{rem}
In the sequel, this stabilization constraint is imposed to ensure the finiteness of the infinite-horizon integrals in~\eqref{cost_inf}. Note that in contrast to LQR, in the LR setting this assumption will impose restrictions on the design of the matrix $E$. 
\begin{defn}[Stationary Feedback Nash Equilibrium]
Consider the dynamic game defined by the system~\eqref{dynamics} and the cost function~\eqref{cost_inf}, for $i=1, \mydots, N$. An admissible set of strategies \scalebox{.94}{$(u_{K_1^*},\mydots, u_{K_N^*})$}, with \scalebox{.94}{$ u_{K_i^*}=-K_i^*x$ and $(K_1^*,\mydots,K_N^*)\in \mathcal K_N$}, constitutes a stationary feedback Nash equilibrium if the following inequalities hold 
\begin{align*}
    \scalebox{.93}{$J_i(x_0, u_{K_1^*},\mydots , u_{K_i^*}, \mydots,  u_{K_N^*})\leq J_i(x_0, u_{K_1^*},\mydots,  u_{K_i}, \mydots, u_{K_N^*})$} 
\end{align*}
for all $i$, $x_0$ and for each admissible \scalebox{.92}{$K_i$}, \scalebox{.94}{$i=1,\mydots, N$} such that \scalebox{.94}{$(K_1^*, \mydots, K_i, \mydots, K_N^*)\in \mathcal K_N$}. 
\end{defn}
Theorem~\ref{thm_inf} states that the feedback Nash equilibria of the LR infinite-horizon problem are characterized,  within the admissible class $\mathcal K_N$ by the solution of the equations
\begin{align}\label{AREs}
    &\scalebox{.94}{$0=s_i^{\top}+ p_i^{\top}A-(r_i^{\top}+ p_i^{\top}B_i)K^*_i-\textstyle\sum\nolimits_{j \neq i} 
    p_i^{\top}B_j K^*_j ,$}
\end{align}
where \scalebox{.94}{$K^*_i \in\hspace{0.5mm} \mathrm{diag}(\mathrm{sign}(r_i+B^{\top}_i p_i))E_i$}. 
\begin{defn}\label{stab_sol}
    A vector tuple \scalebox{.94}{$(\hat p_1,\mydots, \hat p_N)$} with \scalebox{.94}{$\hat p_i\in \mathbb R^n_+$} is called a stabilizing solution of the coupled equations~\eqref{AREs} if each \scalebox{.94}{$\hat p_i$} satisfies~\eqref{AREs} where \scalebox{.94}{$K_i \in\hspace{0.5mm} \mathrm{diag}(\mathrm{sign}(r_i+B_i^{\top}\hat p_i))E_i$} and the matrix \scalebox{.94}{$A-\sum_{i=1}^N B_i K_i$} is Hurwitz. 
\end{defn}
\begin{thm}\label{thm_inf}
Suppose Assumptions~\ref{ass_pos_dyn} and~\ref{ass_nonneg_cost} hold, $(A, B_1, \mydots, B_N)$ is $(E_1,\mydots, E_N)$-stabilizable in the sense of Definition~\ref{Estab} and $(\hat p_1, \mydots, \hat p_N)$, $\hat p_i\in \mathbb R^n_+$ for all $i$ is a set of stabilizing solutions of the algebraic equations~\eqref{AREs} in the sense of Definition~\ref{stab_sol} with $K_i^*\in\hspace{0.5mm}\mathrm{diag}(\mathrm{sign}(r_i+B^{\top}_i \hat p_i))E_i$. Then $(u_{K_1^*}, \mydots, u_{K_N^*})$ is a stationary feedback Nash equilibrium. Moreover, the cost incurred by player $i$ by playing this equilibrium action is $\hat p^{\top}_i x_0$,  $i=1, \mydots, N$. 

Conversely, if $(u_{K_1^*}, \mydots, u_{K_N^*})$ is a stationary feedback Nash equilibrium, then there exists a stabilizing solution $(\hat p_1, \mydots, \hat p_N)$ of the algebraic equations~\eqref{AREs} in the sense of Definition~\ref{stab_sol} such that $K_i^*\in\hspace{0.5mm}\mathrm{diag}(\mathrm{sign}(r_i+B_i^{\top}\hat p_i))E_i$.
\end{thm}
\begin{rem} \label{nonuniqueStat}
   The stationary feedback laws $ u_{K_i^*}=-K^*_i x$ can be nonunique. Indeed, for any player $i$ if there exists a row $j$ such that $(r_i+B_i^{\top}p_i)_j =0$, then  all feedback gains satisfying $-(E_i)_j\leq (K_i)_j\leq (E_i)_j$ lead to different sets of  solutions $(p_1, \mydots, p_N)$ to~\eqref{AREs}, which may yield different stabilizing solutions and hence different Nash equilibria. 
\end{rem}
\begin{proof}
We prove the Theorem for the two-player case, the general case follows analogously.

$\Longrightarrow$ Suppose that Assumptions~\ref{ass_pos_dyn},~\ref{ass_nonneg_cost} hold, \scalebox{.94}{$(A, B_1, \mydots, B_N)$} is \scalebox{.94}{$(E_1,\mydots, E_N)$}-stabilizable and \scalebox{.94}{$(\hat p_1, \hat p_2)$} is a stabilizing solution of the algebraic system of equations~\eqref{AREs}. Fix \scalebox{.94}{$u_2^*=-K_2^*x$} with \scalebox{.94}{$K_2^* \in \hspace{0.5mm} \mathrm{diag}(\mathrm{sign}(r_2+B^{\top}_2 \hat p_2))E_2$} and consider the minimization problem of player 1
\begin{align*}
    \scalebox{.94}{$J_1(x_0, u_1, u_{K_2^*})=\int_0^{\infty} \left( s^{\top}_1 x(t)+ r^{\top}_1 u_1(t) \right)dt$}
\end{align*}
subject to  \scalebox{.92}{$\dot x(t)=(A-B_2K_2^*)x(t)+B_1 u_1(t)$}, $x(0)=x_0$. By assumption, the equation  
\begin{align*}
    \scalebox{.94}{$0=s_1^{\top}+ p^{\top}_1A-(r_1^{\top}+ p^{\top}_1B_1)K_1-p^{\top}_1B_2 K^*_2$}
\end{align*}
has a stabilizing
solution $\hat p_1$. Thus, by Lemma~\ref{LR_inf_thm} this optimal control problem admits a solution. The optimal control law is given by $u_{K_1^*}(t)\hspace{0.5mm} \in \hspace{0.5mm} -\mathrm{diag}(\mathrm{sign}(r_1+ B_1^{\top}\hat p_1))E_1x(t)$ and the corresponding minimum cost is \scalebox{.94}{$\hat p_1^{\top}x_0$}. Hence, \scalebox{.94}{$J_1(x_0, u_{K_1^*}, u_{K_2^*})\leq J_1(x_0, u_{K_1}, u_{K_2^*})$ }for all admissible \scalebox{.94}{$K_1$}. An analogous argument applies to player 2. Therefore, \scalebox{.94}{$(u_{K_1^*}, u_{K_2^*})$} is a stationary feedback Nash equilibrium.

$\Longleftarrow$ Suppose that \scalebox{.94}{$(K_1^*, K_2^*)\in \mathcal K_2$} is a feedback Nash equilibrium. By definition, \scalebox{.94}{$J_1(x_0, u_{K_1}^*, u_{K_2}^*)\leq J_1(x_0, u_{K_1}, u_{K_2}^*)$},  \scalebox{.94}{$J_2(x_0, u_{K_1}^*, u_{K_2}^*)\leq J_2(x_0, u_{K_1}^*, u_{K_2})$} for all \scalebox{.94}{$x_0$} and for all admissible state feedback matrices \scalebox{.94}{$K_1$, $K_2$}, such that \scalebox{.94}{$(K_1^*, K_2)\in \mathcal K_2$} and \scalebox{.94}{$(K_1, K_2^*)\in \mathcal K_2$}. By Lemma~\ref{LR_inf_thm} there exist real nonnegative vectors \scalebox{.94}{$\hat p_1, \hat p_2$}, satisfying the equations
\begin{align}\label{aux1}
    \scalebox{.94}{$0=s_1^{\top}+ \hat p_1^{\top}A-(r_1^{\top}+ \hat p_1^{\top}B_1)K_1-
    \hat p_1^{\top}B_2 K^*_2$}\\
    \scalebox{.94}{$0=s_2^{\top}+ \hat p_2^{\top}A-(r_2^{\top}+ \hat p_2^{\top}B_2)K_2-
    \hat p_2^{\top}B_1 K^*_1$}
    \label{aux2}
\end{align}
such that both $A-B_1K_1-B_2K^*_2$ and $A-B_1K^*_1-B_2K_2$ are Hurwitz. Also by Lemma~\ref{LR_inf_thm}, $K_i \in \hspace{0.5mm}\mathrm{diag}(\mathrm{sign}(r_i+B^{\top}_i \hat p_i))E_i$ and $J_i(x_0, u_{K_1^*}, u_{K_2^*})=\hat p_i^{\top}x_0$, $i=1,2$. Substituting $K_2^*$ in~\eqref{aux1} by $\mathrm{diag}(\mathrm{sign}(r_2+B^{\top}_2 \hat p_2))E_2$ and $K_1^*$ in equation~\eqref{aux2} by $\mathrm{diag}(\mathrm{sign}(r_1+B^{\top}_1 \hat p_1))E_1$, shows that $(\hat p_1, \hat p_2)$ satisfies the coupled set of algebraic equations~\eqref{AREs}. Furthermore, replacing $K_1^*$ by $\mathrm{diag}(\mathrm{sign}(r_1+B^{\top}_1 \hat p_1))E_1$ in $A-B_1K^*_1-B_2K_2$ shows that the matrix $A-B_1K_1-B_2K_2$ is Hurwitz, which completes the proof.
\end{proof}
\begin{alg}\label{rem:fp}
    \vspace{-1mm}
    The infinite-horizon equilibrium~\eqref{AREs} is computed via a fixed-point iteration on the sign structure matrices, initialized with \scalebox{.95}{$S_i^{(0)}=I$}. At iteration \scalebox{.95}{$\ell$}, we construct the matrix \scalebox{.95}{$A_{\mathrm{cl}}^{(\ell)}=A-\sum_{j=1}^{N}B_j S_j^{(\ell)} E_j$}, which is Metzler by Ass \ref{ass_nonneg_cost}. Provided it is Hurwitz (verifiable by the linear program in Remark \ref{stblrem}), we solve the linear system \scalebox{.95}{$p_i^{(\ell+1)}=-(A_{\mathrm{cl}}^{(\ell)})^{-\top} \bigl(s_i-E_i^{\top}S_i^{(\ell)}r_i\bigr)$} and update \scalebox{.95}{$S_i^{(\ell+1)}=\mathrm{diag}\!\bigl(\mathrm{sign}(r_i+B_i^{\top} p_i^{(\ell+1)})\bigr)$}, rows of zero input gradient keep their sign as a selection rule (Remark \ref{nonuniqueStat}). Over finitely many sign structures, the iteration either converges to a selected stationary equilibrium or a cycle is detected.
    \vspace{-0.7mm}
\end{alg}
\section{Simulation Example: Large-Scale Global Pollution Game}
To demonstrate the scalability and rich dynamic behavior of the proposed Linear Regulator framework, we simulate a densely interconnected multinational pollution game. This game is inspired by many classical works in the literature where cross-boundary pollution is modeled to capture the effects of neighboring regions' actions on the common resource's pollution levels \cite{van1991differential, jorgensen2001incentive}, and recent works \cite{wei2020differential, de2021equilibrium}. We consider \scalebox{.95}{$N=3$} strategic players (e.g., multinational companies) interacting over a global environment discretized into \scalebox{.95}{$n=30$} zones, with each player managing \scalebox{.95}{$m_i=20$} industrial sectors. This game employs a dense Metzler diffusion matrix \scalebox{.95}{$A$} (\scalebox{.95}{with diagonal entries of $-1.0$} for localized environmental absorption and off-diagonal entries of \scalebox{.95}{$0.03$} capturing physical spillover) and dense input matrices \scalebox{.95}{$B_i \sim \mathcal{U}(0.1, 1.0)$} (reflecting cross-border industrial footprints). A small constraint matrix \scalebox{.95}{$E_i = 0.0004$} is chosen to model regulatory standards applied to all players. The players are assigned heterogeneous cost profiles where state cost vectors \scalebox{.95}{$s_i >0$} represent environmental cost sensitivities uniformly drawn from player specific ranges, while \scalebox{.95}{$r_i <0$} represent marginal economic profits of industrial emissions linearly spaced between player specific bounds. The resulting finite-horizon dynamics exhibit asynchronous switching between maximum emission and abatement in response to the asymmetric global state and competitors' actions. As shown in Figure \ref{fig:controls}, Player 1, representing a high-yield economy (\scalebox{.95}{$s_1 \in [1.5, 2.5] , r_1 \in [-50,-30]$}), prioritizes profit switching to emissions in the horizon of the game. Player 2 models a vulnerable region with severe environmental penalties (\scalebox{.95}{$s_2 \in [2.0, 4.0] , r_2 \in [-40,-20]$}) with moderate economic margins, which forces a delayed switching closer to end of the game. Finally, Player 3 acts as an emerging economy with low environmental sensivity and economic margins (\scalebox{.95}{$s_3 \in [1.0, 2.0] , r_3 \in [-30,-10]$}), resulting in an intermediate switching regime. Consequently, the global pollution state in Figure \ref{fig:states} initially tracks the infinite-horizon steady-state decay, but exhibit a late-stage resurgence as players switch to maximum emission near the terminal horizon. 
\begin{figure}[htbp]
    \centering
    \includegraphics[scale=0.72]{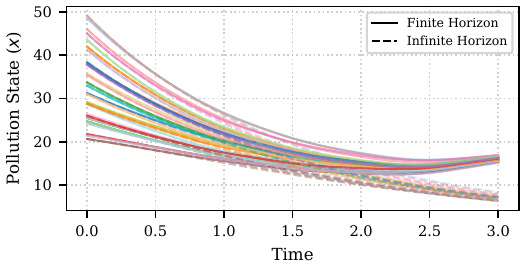}
    \caption{Trajectory of the 30 global pollution states. }
    \label{fig:states}
\end{figure}
\vspace{-15pt}
\begin{figure}[htbp]
    \centering
    \includegraphics[scale=0.72]{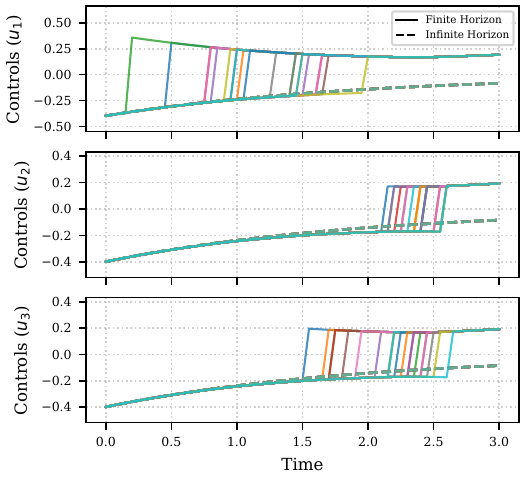}
    \caption{Control inputs across 20 industrial sectors for the three multinational players. Asymmetric rewards and environmental sensitivities trigger staggered, unaligned switching.}
    \label{fig:controls}
\end{figure}
\vspace{-25pt}
\section{Conclusions}
This paper studies a class of continuous-time general-sum non-cooperative games with linear costs, linear dynamics, and elementwise linear input constraints. In the finite-horizon case, under isolated switching regimes, sufficient conditions for feedback Nash equilibria are given by time-varying feedback policies satisfying a system of vector-valued ODEs. In the infinite-horizon case, the equilibria are static feedback stabilizing policies, determined by vector-valued algebraic equations. This work uncovers an interesting dynamic game setting that inherits the scalability potential of the LR framework, while introducing additional and compelling mathematical challenges, such as establishing a priori conditions for the existence, uniqueness, and isolated switching of the coupled ODE solutions, and characterizing the non-isolated switching case, which are natural directions for future research. 
Other directions of future work involve deriving necessary conditions for the existence of feedback Nash equilibria and applying this framework to models of opioid epidemics. 
\vspace{-15pt}
\bibliographystyle{unsrt} 
\bibliography{bibliography}

\end{document}